\documentclass{amsart}
 
\usepackage{mathrsfs, mathtools, amsmath, amssymb, amsthm, bbold}
\usepackage{a4wide}
\usepackage{extpfeil}

\usepackage{hyperref}
\hypersetup{
    colorlinks,
    linkcolor={red!50!black},
    citecolor={blue!50!black},
    urlcolor={blue!80!black}
}

\usepackage{tikz-cd}

\newcommand{\Z}{\mathbb{Z}}
\newcommand{\R}{\mathbb{R}}
\newcommand{\Q}{\mathbb{Q}}
\newcommand{\C}{\mathbb{C}}
\newcommand{\F}{\mathbb{F}}

\usepackage[OT2,T1]{fontenc}
\DeclareSymbolFont{cyrletters}{OT2}{wncyr}{m}{n}
\DeclareMathSymbol{\Sha}{\mathalpha}{cyrletters}{"58}

\newtheorem{theorem}{Theorem}[section]
\newtheorem{proposition}[theorem]{Proposition}
\newtheorem*{proposition*}{Proposition}

\newtheorem{conjecture}[theorem]{Conjecture}

\theoremstyle{definition}

\newtheorem{remark}[theorem]{Remark}

\usepackage{enumitem}

\title[On the Iwasawa theory of elliptic curves at Eisenstein primes]{On the Iwasawa theory of elliptic curves at\\ Eisenstein primes}
\author{Francesc Castella}
\address{Department of Mathematics, University of California, Santa Barbara, CA 93106, USA}
\email{castella@ucsb.edu}

\date{\today}

\begin{document}

\maketitle

\begin{abstract}
These are expanded notes for the mini-course given by the author at the 2022 ICTS workshop `Elliptic curves and the special values of $L$-functions'.
\end{abstract}

\tableofcontents

\newcommand{\fp}{\mathfrak{p}}
\newcommand{\tH}{\text{H}}
\newcommand{\bZ}{\mathbb{Z}}
\newcommand{\bQ}{\mathbb{Q}}
\newcommand{\xrightarrowdbl}[2][]{%
  \xrightarrow[#1]{#2}\mathrel{\mkern-14mu}\rightarrow
}


\section*{Introduction} 


Let $E/\Q$ be an elliptic curve and let $L(E,s)$ be its Hasse--Weil $L$-series. The latter is defined by an Euler product absolutely convergent for complex $s$ in the right-half plane ${\rm Re}(s)>3/2$, but by modularity 
it can be analytically continued to all $s\in\C$.

By the Mordell--Weil theorem, the group of rational points $E(\Q)$ is finitely generated, so
\[
E(\Q)\simeq\Z^{r}\oplus E(\Q)_{\rm tors},
\]
for some $r={\rm rank}_{\Z}E(\Q)\geq 0$. 
The Birch--Swinnerton-Dyer conjecture (BSD) is the statement that
\[
{\rm ord}_{s=1}L(E,s)\overset{?}={\rm rank}_{\Z}E(\Q).
\]
After the groundbreaking works of Coates--Wiles, Rubin, Gross--Zagier, and Kolyvagin in the 1970s and 1980s, the conjecture is known when either $L(E,1)\neq 0$ or $L'(E,1)\neq 0$. In these cases, their results also establish \emph{finiteness} of the Tate--Shafarevich group
\[
\Sha(E/\Q):={\rm ker}\biggl\{{\rm H}^1(\Q,E)\rightarrow\prod_v{\rm H}^1(\Q_v,E)\biggr\},
\]
a statement that is also widely believed to hold in general.

More recently, further progress on the BSD conjecture, and on its refined form predicting an exact formula for the leading Taylor coefficient of $L(E,s)$ around $s=1$ in terms of arithmetic invariants of $E$, has been obtained largely through the use of $p$-adic methods; more specifically, through various incarnations of Iwasawa theory. More specifically, a large body of work has gone into the proof of the following three implications, which are expected to hold for any prime number $p$:

\begin{enumerate}
    \item \emph{$p$-part of the BSD formula in analytic rank $0$}:  
    \[
    L(E,1)\neq 0\quad\Longrightarrow\quad{\rm ord}_p\biggl(\frac{L(E,1)}{\Omega_E}\biggr)={\rm ord}_p\biggl(\frac{\#\Sha(E/\bQ)\cdot{\rm Tam}(E/\Q)}{(\#E(\Q)_{\rm tors})^2}\biggr),
    \]
where $\Omega_E$ is the positive N\'{e}ron period of $E$, and ${\rm Tam}(E/\Q)=\prod_{\ell\mid N}c_\ell(E/\Q)$ is the product of the Tamagawa factors of $E/\Q$. 
    \item \emph{$p$-converse to the theorem of Gross--Zagier and Kolyvagin}:
    \[
    {\rm corank}_{\Z_p}{\rm Sel}_{p^\infty}(E/\Q)=1\quad\Longrightarrow\quad{\rm ord}_{s=1}L(E,s)=1,
    \]
where ${\rm Sel}_{p^\infty}(E/\Q)$ 
is the $p^\infty$-Selmer group fitting into the descent exact sequence
\[
0\rightarrow E(\Q)\otimes\Q_p/\Z_p\rightarrow{\rm Sel}_{p^\infty}(E/\Q)\rightarrow\Sha(E/\Q)[p^\infty]\rightarrow 0.
\]
    \item \emph{$p$-part of the BSD formula in analytic rank $1$}:   
    \[
    {\rm ord}_{s=1}L(E,s)=1\quad\Longrightarrow\quad{\rm ord}_p\biggl(\frac{L'(E,1)}{\Omega_E\cdot{\rm Reg}_E}\biggr)={\rm ord}_p\biggl(\frac{\#\Sha(E/\bQ)\cdot{\rm Tam}(E/\Q)}{(\#E(\Q)_{\rm tors})^2}\biggr),
    \]
    where ${\rm Reg}_E$ is the regulator of the N\'{e}ron--Tate canonical height pairing on $E(\Q)\otimes\R$.
\end{enumerate}

The goal of these lectures is to explain the proof of (1)--(3) for good ordinary primes, with a special emphasis in the case of (the most recently established) \emph{Eisenstein primes $p$}, i.e. primes $p$ for which $E$ admits a rational $p$-isogeny, or equivalently, such that $E[p]$ is reducible as a $G_{\Q}$-module. 
\newline


\emph{Acknowledgements.} It is a pleasure to heartily thank the organizers of the 2022 ICTS workshop `Elliptic curves and the special values of $L$-functions'---Ashay Burungale, Haruzo Hida, Somnath Jha, and Ye Tian---for their invitation to deliver these lectures, and the opportunity to contribute to these proceedings. The author was partially supported by the NSF grant DMS-2101458.


\section{Lecture 1: Main conjectures and applications}

The purpose of this lecture is to explain how, for any \emph{good ordinary prime} (either Eisenstein or not) the implications (1), (2), and (3) from the Introduction follow from certain (three different, but not completely unrelated) ``main conjectures'' in Iwasawa theory. 

\subsection{Mazur's main conjecture}

Let $p>2$ be a good ordinary prime for $E$. Let $\Q(\mu_{p^\infty})$ be the field obtained by adjoining to $\Q$
of $p$-power roots of unity; then 
\[
{\rm Gal}(\Q(\mu_{p^\infty})/\Q)=\Delta\times\Gamma 
\]
with $\Delta\simeq{\rm Gal}(\Q(\mu_p)/\Q)$ a cyclic group of order $p-1$, and $\Gamma\simeq\Z_p$. Let $\Q_\infty/\Q$ be the cyclotomic $\Z_p$-extension of $\Q$, defined as the fixed of $\Q(\mu_{p^\infty})$ by $\Delta$.

For every $n\geq 0$, denote by $\Q_n$ the unique subfield of $\Q_\infty$ with $[\Q_n:\Q]=p^n$. Let ${\rm Sel}_{p^\infty}(E/\Q_n)$ be the usual $p^\infty$-Selmer group, defined as
\[
{\rm Sel}_{p^\infty}(E/\Q_n)={\rm ker}\biggl\{{\rm H}^1(\Q_n,E[p^\infty])\rightarrow\prod_v{\rm H}^1(\Q_v,E)\biggr\},
\]
where $v$ runs over all primes of $\Q$, and put ${\rm Sel}_{p^\infty}(E/\Q_\infty)=\varinjlim_n{\rm Sel}_{p^\infty}(E/\Q_n)$.

The following is a special case of Mazur's control theorem (which applies to abelian varieties defined over a number field $F$ more generally, and arbitrary $\Z_p$-extensions $F_\infty/F$).

\begin{theorem}[Mazur]\label{thm:mazur-control-Q}
The restriction maps
\[
{\rm Sel}_{p^\infty}(E/\Q_n)\rightarrow{\rm Sel}_{p^\infty}(E/\Q_\infty)^{{\rm Gal}(\Q_\infty/\Q_n)}
\]
have finite kernel and cokernel, of order bounded as $n\rightarrow\infty$.
\end{theorem}

The original proof of Theorem~\ref{thm:mazur-control-Q} can be found in \cite{mazur-towers}; an alternative and highly influential proof of the same result is given (for elliptic curves) in \cite{greenberg-cetraro}.


Let $\Lambda=\Z_p[[\Gamma]]$ be the cyclotomic Iwasawa algebra. It follows easily from Theorem~\ref{thm:mazur-control-Q} together with the weak Mordell--Weil theorem, that ${\rm Sel}_{p^\infty}(E/\Q_\infty)$ is cofinitely generated over $\Lambda$, i.e. the Pontryagin dual
\[
X(E/\Q_\infty)
:={\rm Hom}_{\Z_p}({\rm Sel}_{p^\infty}(E/\Q_\infty),\Q_p/\Z_p)
\]
is finitely generated over $\Lambda$. Mazur further conjectured that $X(E/\Q_\infty)$ is $\Lambda$-torsion (see Conjecture~\ref{conj:mazur-IMC} below), a condition that can be easily verified (using a topological version of Nakayama's lemma) when the classical Selmer group ${\rm Sel}_{p^\infty}(E/\Q)$ is finite (so in particular, $E(\Q)$ is finite), but which lies much deeper in general.

On the analytic side, 
using modular symbols (assuming $E$ is parametrized by modular functions) Mazur and Swinnerton-Dyer \cite{M-SwD} attached to $E$ a $p$-adic $L$-function $\mathcal{L}_p^{\rm MSD}(E/\Q)\in\Lambda\otimes\Q_p$ characterized by the property that for every finite order character $\chi:\Gamma\rightarrow\mu_{p^\infty}$ :
\begin{equation}\label{eq:MSD-interp}
\mathcal{L}_p^{\rm MSD}(E/\Q)(\chi)=\begin{cases}
(1-\alpha_p^{-1})^2\cdot\frac{L(E,1)}{\Omega_E}&\textrm{if $\chi=1$,}\\[0.2em]
\frac{p^n}{\tau(\overline{\chi})\alpha_p^n}\cdot\frac{L(E,\overline{\chi},1)}{\Omega_E}&\textrm{if ${\rm cond}(\chi)=p^n>1$,}
\end{cases}
\end{equation}
where $\alpha_p$ is the $p$-adic unit root of $x^2-a_p(E)x+p$ and $\tau(\overline{\chi})$ is the Gauss sum.

Motivated by Iwasawa's main conjecture for class groups of number fields, Mazur formulated the following (see \cite[\S{9.5}, Conj.~3]{M-SwD}). Note that implicit in 
the conjecture is the statement that $\mathcal{L}_p^{\rm MSD}(E/\Q)$ is integral, i.e. lies in $\Lambda$. 

\begin{conjecture}[Mazur's main conjecture]\label{conj:mazur-IMC}
$X(E/\Q_\infty)$ is $\Lambda$-torsion, with
\[
{\rm char}_\Lambda(X(E/\Q_\infty))=\bigl(\mathcal{L}_p^{\rm MSD}(E/\Q)\bigr).
\]
\end{conjecture}


As usual, we identify the Iwasawa algebra $\Lambda$ with the one-variable power series ring $\Z_p[[T]]$ upon the choice of a topological generator $\gamma\in\Gamma$ by setting $T=\gamma-1$. Under this identification, the evaluation  of  an element $\mathcal{L}\in\Lambda$ at a character $\chi:\Gamma\rightarrow\C_p^\times$ corresponds to the specialization of the power series expression of $\mathcal{L}$ at $T=\chi(\gamma)-1$. In particular, evaluation at 
$\chi=1$ corresponds to specialization at $T=0$. 

Henceforth we shall use $a\sim_p b$ to denote the equality $a=ub$ with $u\in\Z_p$.

\begin{proposition}\label{prop:IMC-BSD0}
Assume Conjecture~\ref{conj:mazur-IMC}. Then the $p$-part of the BSD formula holds in analytic rank $0$, i.e.
\[
L(E,1)\neq 0\quad\Longrightarrow\quad{\rm ord}_p\biggl(\frac{L(E,1)}{\Omega_E}\biggr)={\rm ord}_p\biggl(\frac{\#\Sha(E/\bQ)\cdot{\rm Tam}(E/\Q)}{(\#E(\Q)_{\rm tors})^2}\biggr).
\]
\end{proposition}

\begin{proof}
Suppose $L(E,1)\neq 0$. Then $\mathcal{L}_p^{\rm MSD}(E/\bQ)(0)\neq 0$ by the interpolation property. By Mazur's main conjecture, it follows that the $\Gamma$-coinvariants $X(E/\Q_\infty)_\Gamma$ are finite, and so $\#{\rm Sel}_{p^\infty}(E/\Q)<\infty$ by Pontryagin duality and Mazur's control theorem. In particular, $\#{\rm Sel}_{p^\infty}(E/\Q)=\#\Sha(E/\Q)[p^\infty]$.

Let $\mathcal{F}(E/\Q_\infty)\in\Lambda$ be a characteristic power series of $X(E/\Q_\infty)$, i.e. a generator of the principal ideal ${\rm char}_\Lambda(X(E/\Q_\infty))$. Then by the work of Schneider \cite{schneider-II} and Perrin-Riou \cite{PR-92} one has 
\begin{equation}\label{eq:euler-char-Q}
\mathcal{F}(E/\Q_\infty)(0)\sim_p(1-\alpha_p^{-1})^2\cdot\#{\rm Sel}_{p^\infty}(E/\Q)\cdot\frac{{\rm Tam}(E/\Q)}{(\#E(\Q)_{\rm tors})^2}.
\end{equation}
Since by Conjecture~\ref{conj:mazur-IMC} the left-hand side of \eqref{eq:euler-char-Q} has the same $p$-adic valuation as $\mathcal{L}_p^{\rm MSD}(E/\Q)(0)$, the combination of \eqref{eq:MSD-interp} and \eqref{eq:euler-char-Q} yields the result.
\end{proof}



\subsection{Perrin-Riou's main conjecture}

We keep the assumption that $p$ is an odd prime of good ordinary reduction for $E$. 
Let $K/\Q$ be an imaginary quadratic field satisfying the following \emph{Heegner hypothesis}:
\begin{equation}\label{eq:heeg}
\textrm{every prime $\ell\vert N$ splits in $K$.}\tag{Heeg}
\end{equation}
Let $K_\infty^-/K$ be the anticyclotomic $\Z_p$-extension,  characterized as the unique $\Z_p$-extension of $K$ that is Galois over $\Q$ with $\tau\sigma\tau=\sigma^{-1}$ for all $\sigma\in{\rm Gal}(K_\infty^-/K)$, where $\tau$ is the non-trivial automorphism of $K/\Q$. Let $K_n^-$ be the unique subextension of $K_\infty^-$ with $[K_n^-:K]=p^n$. 

Via a fixed modular parametrization
\[
\varphi:X_0(N)\rightarrow E,
\]
the Kummer images of Heegner points of $p$-power conductor
yield classes
\[
x_n\in{\rm Sel}(K_n^-,T_pE):=\varprojlim_m{\rm Sel}_{p^m}(E/K_n^-),
\]
where $T_pE$ is the $p$-adic Tate module of $E$. Using the ordinary hypotheses on $p$, these classes can be made compatible under the corestriction maps ${\rm Sel}(K_{n+1}^-,T_pE)\rightarrow {\rm Sel}(K_n^-,T_pE)$, hence yielding an element
\[
\kappa_\infty^{\rm Hg}\in\check{S}(E/K_\infty^-):=\varprojlim_n{\rm Sel}(K_n^-,T_pE).
\]

Denote by $X(E/K_\infty^-)$ the Pontryagin dual of ${\rm Sel}_{p^\infty}(E/K_\infty^-)$; this  is  a finitely generated module over the anticyclotomic Iwasawa algebra $\Lambda^-=\Z_p[[\Gamma^-]]$, where we put $\Gamma^-={\rm Gal}(K_\infty^-)/K)$.

\begin{conjecture}[Perrin-Riou's main conjecture]\label{conj:PR-IMC}
$X(E/K_\infty^-)$ has $\Lambda^-$-rank $1$, with
\[
{\rm char}_{\Lambda^-}(X(E/K_\infty^-)_{\rm tors})={\rm char}_{\Lambda^-}\biggl(\frac{\check{S}(E/K_\infty^-)}{\Lambda^-\cdot\kappa_\infty^{\rm Hg}}\biggr)^2\cdot\frac{1}{u_K^2c^2},
\]
where the subscript ${\rm tors}$ denotes the maximal $\Lambda^-$-torsion submodule, $u_K:=\frac{1}{2}\#(\mathcal{O}_K^\times)$, and $c\in\Q^\times$ is the Manin constant\footnote{Thus $\varphi^*\omega_E=c\cdot 2\pi if(z)dz$ for the N\'eron differential $\omega_E$ and the newform $f$ attached to $E$.} attached to $\varphi$.
\end{conjecture}

\begin{proposition}\label{prop:HPMC-pGZK}
Assume Conjecture~\ref{conj:PR-IMC}. Then  
\[
{\rm corank}_{\Z_p}{\rm Sel}_{p^\infty}(E/\Q)=1\quad\Longrightarrow\quad{\rm ord}_{s=1}L(E,s)=1,
\]
i.e. the $p$-converse to the theorem of Gross--Zagier and Kolyvagin holds.
\end{proposition}

\begin{proof}
Suppose ${\rm corank}_{\Z_p}{\rm Sel}_{p^\infty}(E/\Q)=1$, and choose an imaginary quadratic field $K$ such that:
\begin{itemize}
\item[(i)] Hypothesis $\eqref{eq:heeg}$ holds;
\item[(ii)] $L(E^K,1)\neq 0$, 
\end{itemize}
where $E^K/\Q$ is the twist of $E$ by the quadratic character corresponding to $K$. By Kato's work \cite{kato-euler-systems}, condition (ii) implies that $\#{\rm Sel}_{p^\infty}(E^K/\Q)<\infty$, and so
\[
{\rm corank}_{\Z_p}{\rm Sel}_{p^\infty}(E/K)={\rm corank}_{\Z_p}{\rm Sel}_{p^\infty}(E/\Q)=1.
\]
By a variant of Theorem~\ref{thm:mazur-control-Q} for the extension $K_\infty^-/K$, it follows that the ${\rm corank}_{\Z_p}(X(E/K_\infty^-)_{\Gamma^-})=1$. By Conjecture~\ref{conj:PR-IMC}, this implies that 
\[
(\gamma-1)\nmid{\rm char}_{\Lambda^-}\biggl(\frac{\check{S}(E/K_\infty^-)}{\Lambda^-\cdot\kappa_\infty^{\rm Hg}}\biggr), 
\]
where $\gamma\in\Gamma^-$ is any topological generator 
(otherwise one would get ${\rm corank}_{\Z_p}(X(E/K_\infty^-)_{\Gamma^-})\geq 3$), and so $\kappa_\infty^{\rm Hg}$ has non-torsion image $\kappa_0^{\rm Hg}$ under the natural map
\[
\check{S}(E/K_\infty^-)\twoheadrightarrow\check{S}(E/K_\infty^-)_{\Gamma^-}\hookrightarrow{\rm Sel}(K,T_pE).
\]
But by construction $\kappa_0^{\rm Hg}$ is the Kummer image of the classical Heegner point $y_K\in E(K)$ in the Gross--Zagier formula \cite{grosszagier}, and therefore $L'(E/K,1)\neq 0$. Finally, the factorization $L(E/K,s)=L(E,s)L(E^K,s)$ together with condition (ii) implies that ${\rm ord}_{s=1}L(E,s)=1$, as desired. 
\end{proof}

\begin{remark}
The first general $p$-converse to the theorem of Gross--Zagier and Kolyvagin for good ordinary primes $p$ is due to Skinner \cite{pCONVskinner} and independently W.\,Zhang \cite{wei-zhang}. The above proof of Proposition~\ref{prop:HPMC-pGZK} is closely related to the approach in \cite{pCONVskinner} and is essentially contained in the work of X.\,Wan \cite{wan-heegner}, which by using the Iwasawa theory of Heegner points (and their ensuing $\Lambda^-$-adic extension of the BDP formula \cite{cas-hsieh1}) allows one to dispense with the assumption  $\#\Sha(E/\Q)[p^\infty]<\infty$ forces upon by the original approach.
\end{remark}


\subsection{BDP main conjecture}

In this section we assume that, in addition to \eqref{eq:heeg}, the imaginary quadratic field $K$ satisfies the condition that
\begin{equation}\label{eq:spl}
\textrm{$(p)=v\overline{v}$ splits in $K$},\tag{spl}
\end{equation}
with $v$ the prime of $K$ above $p$ induced by our fixed embedding $\overline{\Q}\hookrightarrow\overline{\Q}_p$. On the other hand, the condition that $p$ is a prime of good ordinary reduction for $E$ is \emph{not} necessary here.

Put $\Lambda^{\rm ur}:=\Lambda^-\hat\otimes_{\Z_p}\Z_p^{\rm ur}$, where $\Z_p^{\rm ur}$ is the completion of the ring of integers of the maximal unramified extension of $\Q_p$. By the work of Bertolini--Darmon--Prasanna \cite{BDP} and its $\Lambda^-$-adic extension in \cite{brakocevic,cas-hsieh1}, there is a $p$-adic $L$-function $\mathcal{L}_v^{\rm BDP}(f/K)\in\Lambda^{\rm ur}$ characterized by the property that for every character $\chi:\Gamma^-\rightarrow\C_p^\times$ crystalline at both $v$ and $\overline{v}$ of weights $(n,-n)$ with $n>0$ we have
\[
\mathcal{L}_v^{\rm BDP}(f/K)^2(\chi)=C(f/K,\chi)\cdot L^{\rm alg}(f/K,\chi,1),
\]
where $C(f/K,\chi)$ is a nonzero term depending on $f/K$ and $\chi$, and $L^{\rm alg}(f/K,\chi,1)$ is the ``algebraic part'' of the central Rankin--Selberg $L$-value $L(f/K,\chi,1)$.

On the algebraic side, define the \emph{BDP Selmer group} by
\[
{\rm Sel}_v^{\rm BDP}(E/K_\infty^-):={\rm ker}\biggl\{{\rm H}^1(K_\infty^-,E[p^\infty])\rightarrow\prod_{w\nmid v}{\rm H}^1(K_{\infty,w}^-,E[p^\infty])\biggr\}.
\]
In particular, classes in ${\rm Sel}_v^{\rm BDP}(E/K_\infty^-)$ are trivial at the primes above $\overline{v}$. Denote by $X_v^{\rm BDP}(E/K_\infty^-)$ the Pontryagin dual of ${\rm Sel}_v^{\rm BDP}(E/K_\infty^-)$.

The following can be viewed as a special case of Greenberg's Iwasawa main conjectures \cite{greenberg-motives} for $p$-adic deformations of motives.

\begin{conjecture}[BDP main conjecture]\label{conj:BDP-IMC}
$X_v^{\rm BDP}(E/K_\infty^-)$ is $\Lambda^-$-torsion, with
\[
{\rm char}_{\Lambda^-}(X_v^{\rm BDP}(E/K_\infty^-))=\bigl(\mathcal{L}_v^{\rm BDP}(f/K)^2\bigr)
\]
as ideals in $\Lambda^{\rm ur}$.
\end{conjecture}

\begin{proposition}\label{prop:BDP-BSD-1}
Suppose the $p$-part of the BSD formula holds in analytic rank $0$. Then Conjecture~\ref{conj:BDP-IMC} implies the $p$-part of the BSD formula in analytic rank $1$, i.e.
\[
{\rm ord}_{s=1}L(E,s)=1\quad\Longrightarrow\quad{\rm ord}_p\biggl(\frac{L'(E,1)}{\Omega_E\cdot{\rm Reg}_E}\biggr)={\rm ord}_p\biggl(\frac{\#\Sha(E/\bQ)\cdot{\rm Tam}(E/\Q)}{(\#E(\Q)_{\rm tors})^2}\biggr).
\]    
\end{proposition}

\begin{proof}
Suppose ${\rm ord}_{s=1}L(E,s)=1$, and choose an imaginary quadratic field $K$ such that:
\begin{itemize}
\item[(i)] Hypotheses $\eqref{eq:heeg}$ and $\eqref{eq:spl}$ hold;
\item[(ii)] $L(E^K,1)\neq 0$. 
\end{itemize}
Then ${\rm ord}_{s=1}L(E/K,s)=1$, which by the work of Gross--Zagier and Kolyvagin \cite{kolyvagin-mw-sha} implies that the classical Heegner point $y_K\in E(K)$ is non-torsion, and we have
\begin{equation}\label{eq:GZK}
{\rm rank}_{\Z}E(K)=1,\quad\#\Sha(E/K)<\infty;
\end{equation}
in particular, the index $[E(K):\Z y_K]$ is finite. Let $\mathcal{F}_v^{\rm BDP}(E/K_\infty^-)\in\Lambda^-$ be a characteristic power series for $X_v^{\rm BDP}(E/K_\infty^-)$. Then by the work of Jetchev--Skinner--Wan \cite{jsw} we have the equality up to a $p$-adic unit
\begin{equation}\label{eq:JSW}
\mathcal{F}_v^{\rm BDP}(E/K_\infty^-)(0)\sim_p\biggl(\frac{1-a_p(E)+p}{p}\biggr)^2\cdot\prod_{w\mid N}c_w(E/K)\cdot\#\Sha(E/K)\cdot\frac{{\rm log}_{\omega_E}(y_K)^2}{[E(K):\Z y_K]^2},
\end{equation}
where $a_p(E)=p+1-\#E(\F_p)$, $c_w(E/K)$ is the Tamagawa factor of $E$ at $w$, and ${\rm log}_{\omega_E}:E(K_v)\rightarrow\Q_p$ is the formal group logarithm. On the other hand, the formula of Bertolini--Darmon--Prasanna \cite{BDP} yields
\begin{equation}\label{eq:BDP}
\mathcal{L}_p^{\rm BDP}(f/K)^2(0)\sim_p\frac{1}{u_K^2c^2}\cdot\biggl(\frac{1-a_p(E)+p}{p}\biggr)^2\cdot{\rm log}_{\omega_E}(y_K)^2. 
\end{equation}
Since Conjecture~\ref{conj:BDP-IMC} implies that $\mathcal{L}_v^{\rm BDP}(f/K)(0)\sim_p\mathcal{F}_v^{\rm BDP}(E/K_\infty^-)(0)$, combining \eqref{eq:JSW} and \eqref{eq:BDP} we arrive at
\[
[E(K):\Z y_K]^2\sim_p\#\Sha(E/K)\cdot\prod_{w\mid N}c_w(E/K)\cdot u_K^2c^2.
\]
By Gross--Zagier formula \cite{grosszagier}, this last relation is \emph{equivalent} to the $p$-part of the BSD formula when ${\rm ord}_{s=1}L(E/K)=1$. Thus using from the factorization  
\[
L(E/K,s)=L(E,s)L(E^K,s)
\]
and the assumption that the $p$-part of the BSD formula holds for $L(E^K,1)$, the result follows.
\end{proof}

\section{Lecture 2: BDP main conjecture at Eisenstein primes}

\subsection{Main result}\label{subsec:lec3-main}

Let $p\nmid 2N$ be a prime of good ordinary reduction for $E$. When the residual representation 
\[
\rho_{E,p}:G_{\Q}\rightarrow{\rm Aut}_{\F_p}(E[p])\simeq{\rm GL}_2(\F_p)
\]
has ``big image'' (and satisfies some mild ramification hypotheses), Conjectures~\ref{conj:PR-IMC} and \ref{conj:BDP-IMC} are known by combining:
\begin{itemize}
\item Euler/Kolyvagin system methods using Heegner points (\cite{mazrub},  \cite{howard});
\item A vast generalization of Ribet's methods (\cite{skinner-urban}, \cite{wan-ANT-IMC,wan-HPMC}).
\end{itemize}
Now we put ourselves in the opposite case where $E[p]$ is \emph{reducible} as a $G_\Q$-module, say
\begin{equation}\label{eq:Eis-congr}
E[p]^{ss}\simeq\F_p(\phi)\oplus\F_p(\psi),
\end{equation}
where $\phi,\psi:G_\Q\rightarrow\F_p^\times$ are characters. Note that $\psi=\omega\phi^{-1}$ by the Weil pairing, where $\omega:G_\Q\rightarrow\F_p^\times$ is the mod $p$ cyclotomic character. The goal of this lecture is to outline the proof of the following result from \cite{eisenstein} (in the rank one case) and \cite{CGS}. 

\begin{theorem}\label{thm:CGLS}
Let $K$ be an imaginary quadratic field of odd discriminant $-D_K\neq -3$, and satisfying hypotheses \eqref{eq:heeg} and \eqref{eq:spl}. Suppose $p>2$ is a good Eisenstein prime for $E$ with 
\[
\phi\vert_{G_p}\neq 1,\omega,
\]
where $G_p\subset G_\Q$ is a decomposition group at $p$. Then the BDP main conjecture (Conjecture~\ref{conj:BDP-IMC}) and Perrin-Riou's main conjecture (Conjecture~\ref{conj:PR-IMC}) both hold.
\end{theorem}

Recall that $\Lambda^-$ denotes the anticyclotomic Iwasawa algebra. From the structure theorem for finitely generated $\Lambda^-$-modules and the Weierstrass preparation theorem, one has Iwasawa $\lambda$- and $\mu$-invariants attached to $X_v^{\rm BDP}(E/K_\infty^-)$ and $\mathcal{L}_v^{\rm BDP}(E/K)$. An understanding of these invariants is a key in Theorem~\ref{thm:CGLS}, whose proof is naturally divided into 2 steps:

\begin{itemize}
\item{}\emph{Step 1}: Exploit the congruence \eqref{eq:Eis-congr} to show that 
\begin{align*}
\mu(X_v^{\rm BDP}(E/K_\infty^-))&=\mu(\mathcal{L}_v^{\rm BDP}(E/K))=0,\\
\lambda(X_v^{\rm BDP}(E/K_\infty^-))&=\lambda(\mathcal{L}_v^{\rm BDP}(E/K)^2).
\end{align*}
\item{}\emph{Step 2}: Show that $X_v^{\rm BDP}(E/K_\infty^-)$ is $\Lambda^-$-torsion, with 
\[
{\rm char}_{\Lambda^-}(X_v^{\rm BDP}(E/K_\infty^-))\supset\bigl(\mathcal{L}_v^{\rm BDP}(f/K)^2\bigr)
\]
as ideals in $\Lambda^{\rm ur}[1/p]$.
\end{itemize}

Clearly the combination of these two imply the equality 
\[
{\rm char}_{\Lambda^-}(X_v^{\rm BDP}(E/K_\infty^-))=\bigl(\mathcal{L}_v^{\rm BDP}(f/K)^2\bigr)
\]
in $\Lambda^{\rm ur}$ predicted by Conjecture~\ref{conj:BDP-IMC}. That they also imply Conjecture~\ref{conj:PR-IMC} follows from the \emph{equivalence} between the two conjectures, a consequence of the $\Lambda^-$-adic analogue of the BDP formula \cite{BDP} obtained in \cite{cas-hsieh1}. 

\begin{remark}
In a recent work \cite{Keller-Yin}, T.\,Keller and M.\,Yin have removed the hypothesis on $\phi$ in Theorem~\ref{thm:CGLS}. They have also extended the result to higher weight modular forms, and (using Hida theory in the style of Skinner \cite{skinner-mult}) even to the  case of multiplicative Eisenstein primes.
\end{remark}

In the next two subsections we outline the main ideas that go into  the proofs of the above \emph{Step~1} and \emph{Step~2}, respectively.

\subsection{Anticyclotomic Greenberg--Vatsal method}\label{subsec:GV}

Denote by $S$ the set of primes of $K$ dividing $N$, and by $\Sigma\supset S$ the set of primes of $K$ dividing $Np\infty$. Let $K^\Sigma$ be the Galois group of the maximal extension of $K$ unramified outside $\Sigma$, and consider the \emph{$S$-imprimitive} BDP Selmer group
\begin{equation}\label{eq:Sel-E}
{\rm Sel}_{v,S}^{\rm BDP}(E/K_\infty^-):={\rm ker}\biggl\{{\rm H}^1(K^\Sigma/K_\infty^-,E[p^\infty])\rightarrow\prod_{w\mid\overline{v}}{\rm H}^1(K_{\infty,w}^-,E[p^\infty])\biggr\}.
\end{equation}
Let $X^{\rm BDP}_{v,S}(E/K_\infty^-)$ be the Pontryagin dual of ${\rm Sel}_{v,S}^{\rm BDP}(E/K_\infty^-)$. Multiplying $\mathcal{L}_v^{\rm BDP}(f/K)$ by certain elements in $\Lambda^-$ interpolating the local Euler factors of $L(f/K,\chi,s)$ at $s=1$ at primes $v\in S$ over characters $\chi$ of $\Gamma^-$, one can define an $S$-imprimitive $\mathcal{L}_{v,S}^{\rm BDP}(f/K)\in\Lambda^{\rm ur}$ interpolating the central $L$-values of $L(f/K,\chi,s)$ at $s=1$ with the Euler factors at the primes in $S$ stripped out.

The principle to be exploited is that Conjecture~\ref{conj:BDP-IMC} should be equivalent to its $S$-imprimitive counterpart, so in particular
\[
{\rm char}_{\Lambda^-}(X_{v,S}^{\rm BDP}(E/K_\infty))\overset{?}=\bigl(\mathcal{L}_{v,S}^{\rm BDP}(f/K)^2\bigr),
\]
with the latter having the advantage (first noticed by Greenberg in the context of classical Iwasawa theory \cite{greenberg-nagoya}) that the objects involved  are better-behaved with respect to congruences.

Let $\Phi,\Psi:G_\Q\rightarrow\Z_p^\times$ be the Teichm\"uller lifts of $\phi,\psi$, respectively. Attached to $\Phi,\Psi$ one has $\Lambda^-$-cotorsion Selmer groups ${\rm Sel}_{v,S}(\Phi/K_\infty^-),{\rm Sel}_{v,S}(\Psi/K_\infty^-)$  (whose definition is recalled in the proof of Proposition~\ref{prop:Iw-inv-alg} below) with associated Iwasawa $\lambda$-invariants denoted $\lambda_\phi^S,\lambda_\psi^S$, respectively.

\begin{proposition}\label{prop:Iw-inv-alg}
Suppose $p\nmid 2N$ is such that $E[p]^{ss}\simeq\F(\phi)\oplus\F(\psi)$ as $G_\Q$-modules with $\phi\vert_{G_p}\neq 1,\omega$. Then $X_{v,S}^{\rm BDP}(E/K_\infty^-)$ is $\Lambda^-$-torsion, with
\[
\mu(X_{v,S}^{\rm BDP}(E/K_\infty^-))=0,\quad\quad
\lambda(X_{v,S}^{\rm BDP}(E/K_\infty^-))=\lambda_\phi^S+\lambda_\psi^S.
\]
\end{proposition}

\begin{proof}
Let $K_\phi$ is the fixed field of ${\rm ker}(\phi\vert_{G_K})$, and let $M_\infty$ be the maximal abelian pro-$p$ extension of $K_\infty^-K_\phi$ unramfied outside $v$ and $S$. By standard arguments, the Selmer group
\begin{equation}\label{eq:Sel-phi}
\begin{aligned}
{\rm Sel}_{v,S}(\Phi/K_\infty^-)&:={\rm ker}\biggl\{{\rm H}^1(K^\Sigma/K_\infty^-,\Q_p/\Z_p(\Phi))\rightarrow\prod_{w\mid\overline{v}}{\rm H}^1(K_{\infty,w}^-,\Q_p/\Z_p(\Phi))\biggr\}\\
&\,\simeq{\rm Hom}_{\rm cts}({\rm Gal}(M_\infty/K_\infty^-K_\phi),\Q_p/\Z_p)
\end{aligned}
\end{equation}
is $\Lambda^-$-cotorsion and with no proper $\Lambda^-$-submodules of finite index. On the other hand, by Hida's result on the vanishing of the $\mu$-invariant of anticyclotomic Katz $p$-adic $L$-functions \cite{hidamu=0} together with Rubin's proof of the Iwasawa main conjecture for $K$ \cite{rubinmainconj}, we have $\mu({\rm Sel}_{v,S}(\Phi/K_\infty^-)^\vee)=0$. Thus we see that ${\rm Sel}_{v,S}(\Phi/K_\infty^-)$ is $p$-divisible, and therefore the $\lambda$-invariant of its Pontryagin dual ${\rm Sel}_{v,S}(\Phi/K_\infty^-)^\vee$ is given by
\begin{equation}\label{eq:lambda}
\lambda_\phi^S={\rm dim}_{\F_p}\bigl({\rm Sel}_{v,S}(\Phi/K_\infty^-)[p]\bigr).
\end{equation}
From our conditions on $\phi$, it is easy to see that the natural map
\[
{\rm H}^1(K_\infty^-,\Q_p/\Z_p(\phi))\rightarrow{\rm H}^1(K_\infty^-,\Q_p/\Z_p(\Phi))[p]
\]
gives ${\rm Sel}_{v,S}(\phi/K_\infty^-)\simeq{\rm Sel}_{v,S}(\Phi/K_\infty^-)[p]$, where ${\rm Sel}_{v,S}(\phi/K_\infty^-)$ is the \emph{residual Selmer group} defined as in \eqref{eq:Sel-phi}  with $\F_p(\phi)$ in place of $\Q_p/\Z_p(\Phi)$. Of course, the same results apply with
$\psi=\omega\phi^{-1}$ in place of $\phi$. 

Letting ${\rm Sel}_{v,S}^{\rm BDP}(E[p]/K_\infty^-)$ be the Selmer group defined as in \eqref{eq:Sel-E} with $E[p^\infty]$ replaced by $E[p]$, from the short exact sequence
\begin{equation}\label{eq:Eis-congr-ses}
0\rightarrow\F_p(\phi)\rightarrow E[p]\rightarrow\F_p(\psi)\rightarrow 0
\end{equation}
we immediately arrive at the short exact sequence
\begin{equation}\label{eq:Sel-ses}
0\rightarrow{\rm Sel}_{v,S}(\phi/K_\infty^-)\rightarrow{\rm Sel}_{v,S}^{\rm BDP}(E[p]/K_\infty^-)\rightarrow{\rm Sel}_{v,S}(\psi/K_\infty^-)\rightarrow 0.
\end{equation}

The above thus shows that ${\rm Sel}_{v,S}(E[p]/K_\infty^-)\simeq{\rm Sel}_{v,S}^{\rm BDP}(E/K_\infty^-)[p]$ is finite, and so $X_{v,S}^{\rm BDP}(E/K_\infty^-)$ is $\Lambda^-$-torsion with $\mu=0$. Since similarly as before the $\lambda$-invariant of ${\rm Sel}_{v,S}^{\rm BDP}(\Phi/K_\infty^-)^\vee$ can be computed as ${\rm dim}_{\F_p}({\rm Sel}_{v,S}^{\rm BDP}(E/K_\infty^-)[p])$, the last claim in the proposition follows from \eqref{eq:Sel-ses} and \eqref{eq:lambda}.
\end{proof}

On the analytic side, \eqref{eq:Eis-congr-ses} implies a congruence
\[
f\equiv E_{\phi,\psi}\pmod{p}
\]
between the newform $f$ attached to $E$ 
and a weight $2$ Eisenstein series $E_{\phi,\psi}$ attached to the Dirichlet characters $\Phi,\Psi$. From the constructions of  $\mathcal{L}_{v,S}^{\rm BDP}(f/K)$ and of the Katz $p$-adic $L$-function for characters of $K$ \cite{katz-CM,HT-ENS}, building on work of Kriz \cite{kriz} one then deduces a congruence
\[
\mathcal{L}_{v,S}^{\rm BDP}(E/K)^2\equiv\mathcal{L}_{v,S}^{\rm Katz}(\Phi)\cdot\mathcal{L}_{v,S}^{\rm Katz}(\Psi)\pmod{p\Lambda^{\rm ur}},
\]
which together with the aforementioned vanishing result of Hida  yields the equalities
\begin{equation}\label{eq:Iw-an}
\mu(\mathcal{L}_{v,S}^{\rm BDP}(E/K))=0,\quad\quad\lambda(\mathcal{L}_{v,S}^{\rm BDP}(E/K)^2)=\lambda(\mathcal{L}^{\rm Katz}_{v,S}(\Phi))+\lambda(\mathcal{L}^{\rm Katz}_{v,S}(\Psi)).\nonumber
\end{equation}
By Rubin's proof of the Iwasawa main conjecture for $K$, these last two equalities and Proposition~\ref{prop:Iw-inv-alg} yield the proof of \emph{Step~1}.

\subsection{Kolyvagin system argument with ``error terms''}\label{subsec:Kolyvagin}

As noted in $\S\ref{subsec:lec3-main}$, the proof of Theorem~\ref{thm:CGLS} exploits the following interplay  between Conjectures~\ref{conj:BDP-IMC} and Conjecture~\ref{conj:PR-IMC}. 

\begin{proposition}\label{prop:equiv}
Suppose $E(K)[p]=0$. Then the following are equivalent:
\begin{enumerate}
\item 
$X_v^{\rm BDP}(E/K_\infty^-)$ is $\Lambda^-$-torsion, $\mathcal{L}_v^{\rm BDP}(f/K)$ is nonzero, and
\[
{\rm char}_{\Lambda^-}(X_v^{\rm BDP}(E/K_\infty^-))\supset\bigl(\mathcal{L}_v^{\rm BDP}(f/K)^2\bigr)
\]
in $\Lambda^{\rm ur}[1/p]$.
\item 
$X(E/K_\infty^-)$ has $\Lambda^-$-rank one, $\kappa_\infty^{\rm Hg}$ is not $\Lambda^-$-torsion, and
\[
{\rm char}_{\Lambda^-}(X(E/K_\infty^-)_{\rm tors})\supset{\rm char}_{\Lambda^-}\biggl(\frac{\check{S}(E/K_\infty^-)}{\Lambda^-\cdot\kappa_\infty^{\rm Hg}}\biggr)^2
\]
in $\Lambda^-[1/p]$.
\end{enumerate}
The same result holds  for the opposite divisibilities, and without inverting $p$.
\end{proposition}

\begin{proof}[Sketch of proof]
By $p$-ordinarity, there is a unique quotient $T_p^-E\simeq\Z_p$ of $T_pE$  where the $G_{p}$-action is unramified.  From the two-variable extension (due to Loeffler--Zerbes \cite{LZ-IJNT}) of the cyclotomic Perrin-Riou big logarithm map \cite{PR-94} one can deduce the existence of an injective generalized Coleman power series map with pseudo-null cokernel
\[
{\rm Col}_v:\varprojlim_n{\rm H}^1(K_{n,v}^-,T_p^-E)\hookrightarrow\Lambda^{\rm ur},
\]
which by virtue of a $\Lambda^-$-adic extension of the BDP formula (see \cite{cas-hsieh1}) sends the natural image of ${\rm res}_v(\kappa_\infty^{\rm Hg})$ to $\mathcal{L}_v^{\rm BDP}(f/K)$. The result then follows from a double application (one involving ${\rm res}_v$ and another involving ${\rm res}_{\overline{v}}$) of Poitou--Tate duality. 
\end{proof}

Since the fact that $\kappa_\infty^{\rm Hg}$ is not $\Lambda^-$-torsion follows from the work of Cornut--Vatsal \cite{cornut,vatsal}\footnote{Alternatively, it also follows from the $\Lambda^-$-adic BDP formula and the nonvanishing of $\mathcal{L}_v^{\rm BDP}(f/K)$ (see \cite{hsieh-special}) via Hida's methods.}, the proof of \emph{Step~2}, and hence of Theorem~\ref{thm:CGLS}, is thus reduced to the following.

\begin{proposition}\label{prop:KS-1}
Suppose $E(K)[p]=0$. Then $X(E/K_\infty^-)$ has $\Lambda^-$-rank one, and we have 
\[
{\rm char}_{\Lambda^-}(X(E/K_\infty^-)_{\rm tors})\supset{\rm char}_{\Lambda^-}\biggl(\frac{\check{S}(E/K_\infty^-)}{\Lambda^-\cdot\kappa_\infty^{\rm Hg}}\biggr)^2
\]
in $\Lambda^-[1/p]$.
\end{proposition}

\begin{proof}
This follows from a refinement of Kolyvagin's methods building on some of the techniques developed by Howard and Nekov\'{a}\v{r} (see \cite{howard,nekovar}) in related settings. The difficulty in the present case lies in the fact that no ``big image'' hypotheses on $T_pE$ is being made.

By standard arguments, the non-triviality of $\kappa_\infty^{\rm Hg}$ and a generalized Cassels--Tate pairing implies the existence of a $\Lambda^-$-module pseudo-isomorphism
\[
X(E/K_\infty^-)\sim\Lambda^-\oplus M\oplus M
\]
with $M$ a finitely generated torsion $\Lambda^-$-module. Thus the task is to compare the characteristic ideal of $M$ with that of ${\check{S}(E/K_\infty^-)}/\Lambda^-\cdot\kappa_\infty^{\rm Hg}$. Let $\mathfrak{P}$ be a height one prime of $\Lambda^-$ with $\mathfrak{P}\neq(p)$, and take a sequence $\mathfrak{P}_m$ of height one primes of $\Lambda^-$ with $\mathfrak{P}_m\to\mathfrak{P}$ as $m\to\infty$. Note that each such $\mathfrak{P}_m$ corresponds to a character $\alpha_m:\Gamma^-\rightarrow R_m^\times$ with $R_m$ a finite extension of $\Z_p$. By inductively choosing a sequence of Kolyvagin primes (of ``depth $k$'' for $k\gg 0$) using Cebotarev, one arrives at the inequality
\[
{\rm length}_{R_m}(M_{\mathfrak{P}_m})\leq{\rm length}_{R_m}\bigl(\check{S}(E/K_\infty)_{\mathfrak{P}_m}/R_m\cdot\kappa_{\infty,\mathfrak{P}_m}^{\rm Hg}\bigr)+E_m,
\]
where $E_m$ is an ``error term'' behaving asymptotically like ${\rm ord}_p(\alpha_m(\gamma)-\alpha_m^{-1}(\gamma))$ as $m\to \infty$.   
Thus $E_m=O(1)$ as long as $\mathfrak{P}\neq (\gamma-1)$, and hence by a control theorem in the style of Mazur--Rubin \cite{mazrub}, letting $\mathfrak{P}$ vary we deduce that 
the claimed divisibility holds in $\Lambda^-[1/p,1/(\gamma-1)]$. To handle the prime $\mathfrak{P}=(\gamma-1)$, one takes a sequence $\mathfrak{P}_m$ with $\alpha_m\equiv 1\pmod{p^m}$, and choosing a sequence of Kolyvagin primes as above, but this time exploiting the action of complex conjugation on  $(T_pE\otimes\alpha_m)/p^m$, a different induction argument 
yields the inequality
\[
{\rm length}_{R_m}(M_{\mathfrak{P}_m})\leq{\rm length}_{R_m}\bigl(\check{S}(E/K_\infty)_{\mathfrak{P}_m}/R_m\cdot\kappa_{\infty,\mathfrak{P}_m}^{\rm Hg}\bigr)+E_m,
\]
with an error term $E_m$ now bounded independently of $m$, which by a control theorem yields the desired divisibility also at the augmentation ideal $(\gamma-1)$.
\end{proof}

\begin{remark} For the application to the $p$-converse to the theorem of Gross--Zagier and Kolyvagin, it suffices to have the divisibility ``$\subset$'' in Theorem~\ref{thm:CGLS} (rather than the equality of characteristic ideals) after inverting $(\gamma-1)$ and $(p)$;  similarly, an ambiguity by powers of $(\gamma-1)$ is harmless for the application to the $p$-part of the BSD formula in analytic rank one.  However, the final from of the result of Theorem~\ref{thm:CGLS} obtained in \cite{CGS} is essential to the proof of Mazur's main conjecture at Eisenstein primes explained in the next lecture.
\end{remark}

\section{Lecture 3: Mazur's main conjecture at Eisenstein primes}

\subsection{Main result}

In this lecture we explain the proof of the following result from \cite{CGS}.

\begin{theorem}\label{thm:CGS}
Let $E/\Q$ be an elliptic curve of conductor $N$, and let $p\nmid 2N$ be a good Eisenstein prime for $E$, i.e. such that
\[
E[p]^{ss}\simeq\F_p(\phi)\oplus\F_p(\psi)
\]
for characters $\phi,\psi=\omega\phi^{-1}:G_{\Q}\rightarrow\F_p^\times$. 
Assume that $\phi\vert_{G_p}\neq 1,\omega$.  
Then Mazur's main conjecture (Conjecture~\ref{conj:mazur-IMC}) holds for $E$.
\end{theorem}

Previously, the following results were known towards Conjecture~\ref{conj:mazur-IMC} for good Eisenstein primes $p$:
\begin{itemize}
    \item Rubin \cite{rubinmainconj}: proof in the CM case.
    \item Kato \cite{kato-ES}: $X(E/\Q_\infty)$ is $\Lambda$-torsion, with
    \[
    {\rm char}_\Lambda(X(E/\Q_\infty))=\bigl(\mathcal{L}_p^{\rm MSD}(E/\Q)\bigr)
    \]
    in $\Lambda[1/p]$.
    \item W\"uthrich \cite{wuthrich-int}: $\mathcal{L}_p^{\rm MSD}(E/\Q)$ is integral, and Kato's divisibility holds in $\Lambda$.
    \item Greenberg--Vatsal \cite{greenvats}: proof in ``half'' of the cases; more precisely, when
    \begin{equation}\label{eq:GV}\tag{GV}
    \begin{aligned}
    \phi=\begin{cases}\textrm{unramified at $p$ and odd, or}\\[0.2em]
    \textrm{ramified at $p$ and even;}
    \end{cases}
    \end{aligned}
    \end{equation}
    in other words, when $E[p^\infty]$ contains no cyclic subgroups of multiplicative type.
\end{itemize}
The condition on $\phi$ in the Greenberg--Vatsal result is needed to ensure the vanishing of $\mu(X(E/\Q_\infty))$ building on the work of Ferrero--Washington \cite{FW} and Mazur--Wiles \cite{mazur-wiles}.    
Without this restriction on $\phi$, it was shown by Greenberg \cite{greenberg-cetraro} that $\mu(X(E/\Q_\infty))$ is positive, and by work of Stevens \cite{stevens-invmath} one similarly knows that $\mu(\mathcal{L}_p^{\rm MSD}(E/\Q))>0$ when $\phi$ doesn't satisfy \eqref{eq:GV}. 

Thus to extend the  Greenberg--Vatsal method beyond the cases covered by \eqref{eq:GV} one is faced with the challenge of determining the exact value of the algebraic and analytic invariants, which seems to be a very difficult problem (but see \cite{bellaiche-pollack} and \cite{pollack-wake} for interesting recent works in this direction). 

The proof of Theorem~\ref{thm:CGS} is based on a different method to compare Iwasawa invariants. 
The method is insensitive to the value of $\mu$, and in particular gives a new proof of the Greenberg--Vatsal result in the cases they considered.

\subsection{Comparing Iwasawa invariants}

In this section we explain the strategy from \cite{CGS} to arrive at the equalities
\begin{equation}\label{eq:Iw-inv-CGS}
\mu(X(E/\Q_\infty))=\mu(\mathcal{L}_p^{\rm MSD}(E/\Q)),\quad
\lambda(X(E/\Q_\infty))=\lambda(\mathcal{L}_p^{\rm MSD}(E/\Q)),
\end{equation}
which combined with Kato's divisibility (as integrally refined by W\"uthrich \cite{wuthrich-int}) yields Theorem~\ref{thm:CGS}. Some of the details on how the strategy is carried out are given in the next subsection.

The following discussion applies to any prime $p\nmid 2N$ of good ordinary reduction for $E$. 
Let $K$ be an imaginary quadratic field satisfying \eqref{eq:spl}, and let $K_\infty^+$ be the cyclotomic $\Z_p$-extension of $K$. Following Greenberg \cite{greenberg-reps}, we define the ordinary Selmer group of $E$ over $K_\infty^+$ by
\[
{\rm Sel}_{p^\infty}(E/K_\infty^+):=\ker\biggr\{{\rm H}^1(K_\infty^+,E[p^\infty])\rightarrow\prod_{w\mid p}\frac{{\rm H}^1(K_{\infty,w}^+,E[p^\infty])}{A_w}\times\prod_{w\nmid p}{\rm H}^1(I_w,E[p^\infty])\biggr\},
\]
where $A_w:={\rm im}\{E^+[p^\infty]\rightarrow E[p^\infty]\}_{\rm div}$, with $E^+[p^\infty]$ the kernel of the reduction map at $p$, and $I_w\subset G_{K_{\infty,w}^+}$ the inertia subgroup at $w$. On the analytic side, Hida's $p$-adic Rankin method \cite{hida-measure-I} (as studied by Perrin-Riou \cite{PR-JLMS} in detail in the case of Rankin--Selberg convolution of $f$ with theta series of $K$) yields the construction of a $2$-variable $p$-adic $L$-function 
\[
\mathcal{L}_p^{\rm PR}(E/K)\in\Lambda_K:=\Z_p[[{\rm Gal}(K_\infty/K)]],
\]
where $K_\infty/K$ is the $\Z_p^2$-extension of $K$, interpolating  the algebraic part of the central $L$-values $L(f/K,\chi,1)$ (with a normalized period depending on $E$), as $\chi$ runs over the finite orders characters of $\Gamma_K$. 

The action of complex conjugation yields a decomposition $\Gamma_K\simeq\Gamma^+\times\Gamma^-$ into $\pm$-eigenspaces, with $\Gamma^+$ (resp. $\Gamma^-$) idenfitied with the Galois group of the cyclotomite (resp. anticyclotomic) $\Z_p$-extension of $K$.  
Denoting by $\mathcal{L}_p^{\rm PR}(E/K)^+$ the image of $\mathcal{L}_p^{\rm PR}(E/K)$ under the natural projection 
\[
\Lambda_K\rightarrow\Lambda^+:=\Z_p[[{\rm Gal}(K_\infty^+/K)]]\simeq\Lambda,
\]
Greenberg's Iwasawa Main Conjecture for general $p$-ordinary representations \cite{greenberg-reps}
predicts that for $\star\in\{+,\emptyset\}$, the Pontryagin dual $X(E/K_\infty^\star)={\rm Hom}_{\Z_p}({\rm Sel}_{p^\infty}(E/K_\infty^\star),\Q_p/\Z_p)$ is $\Lambda^\star$-torsion, with 
\begin{equation}\label{eq:PR+}
{\rm char}_{\Lambda^\star}(X(E/K_\infty^\star))\overset{?}=\bigl(\mathcal{L}_p^{\rm PR}(E/K)^\star\bigr).
\end{equation}

As a motivation for the general argument, we 
note that the aforementioned results, together with Theorem~\ref{thm:CGLS}, already imply a proof of this conjecture in some cases. Indeed,  denote by $E^K$ the twist of $E$ by the quadratic character corresponding to $K$. Kato's integral divisibility towards  Conjecture~\ref{conj:mazur-IMC} for $E$ and $E^K$ yields the divisibility 
\begin{equation}\label{eq:kato-K}
{\rm char}_{\Lambda^+}(X(E/K_\infty^+))\supset \bigl(\mathcal{L}_p^{\rm PR}(E/K)^+\bigr),
\end{equation}
while from Theorem~\ref{thm:CGLS} and the fact that $K_\infty^-\cap K_\infty^+=K$ one can show the equality up to a $p$-adic unit
\begin{equation}\label{eq:0}
\mathcal{F}(E/K_\infty^+)(0)\sim_p\mathcal{L}_p^{\rm PR}(E/K)^+(0),
\end{equation}
where $\mathcal{F}(E/K_\infty^+)\in\Lambda^+$ is any characteristic power series for $X(E/K_\infty^+)$.  
It is easy to see that the combination of \eqref{eq:kato-K} and \eqref{eq:0} implies \eqref{eq:PR+}, and hence Conjecture~\ref{conj:mazur-IMC}, \emph{provided} $\mathcal{L}_p^{\rm PR}(E/K)^+(0)\neq 0$. Unfortunately, hypothesis \eqref{eq:heeg} forces this value to vanish for sign reasons. Using Beilinson--Flach classes and their explicit reciprocity laws (as described in more detail in the next subsection), the same conclusion applies \emph{provided} $\mathcal{L}_v^{\rm BDP}(E/K)(0)\neq 0$, which by the main result of \cite{BDP} amounts to the requirement that the Heegner point $y_K\in E(K)$ is non-torsion.

To treat the general case, the idea is to take an anticyclotomic character 
\[
\alpha:\Gamma^-\rightarrow\Z_p^\times
\]
with $\alpha\equiv 1\pmod{p^M}$, for some $M\gg 0$ to stay away from any problematic zeroes; in particular, so that $\mathcal{L}_v^{\rm BDP}(E/K)(\alpha)\neq 0$. From a refinement \cite{BST} of 
the Beilinson--Flach classes constructed by Lei--Loeffler--Zerbes \cite{LLZ-RS,LLZ-K} and Kings--Loeffler--Zerbes \cite{KLZ-AJM,explicit} (in particular allowing one of the forms used in the construction to be a residually reducible and $p$-indistinguished Hida family with CM by $K$), 
and their explicit reciprocity laws, one can deduce from Theorem~\ref{thm:CGLS} a proof of the $\alpha$-twisted variant of conjecture \eqref{eq:PR+} for $K_\infty^+/K$:
\begin{equation}\label{eq:main-result}
{\rm char}_{\Lambda^+}(X(E(\alpha)/K_\infty^+))\overset{?}=\bigl(\mathcal{L}_p^{\rm PR}(E(\alpha)/K)^+\bigr).
\end{equation}
Establishing \eqref{eq:main-result} for a suitable choice of $\alpha$ as above is the key to the proof of Theorem~\ref{thm:CGS}, since from the easy congruences 
\begin{align*}
{\rm char}_{\Lambda^+}(X(E(\alpha)/K_\infty^+))&\equiv{\rm char}_{\Lambda^+}(X(E/K_\infty^+))\pmod{p^M},\\
\mathcal{L}_p^{\rm PR}(E(\alpha)/K)^+&\equiv\mathcal{L}_p^{\rm PR}(E/K)^+\pmod{p^M},
\end{align*}
it implies the equalities 
\[
\mu(X(E/K_\infty^+))=\mu(\mathcal{L}_p^{\rm PR}(E/K)^+),\quad
\lambda(X(E/K_\infty^+))=\lambda(\mathcal{L}_p^{\rm PR}(E/K)^+)
\]
(in particular, without knowing the specific value of the $\mu$-invariants!). 
Together with the integral divisibility \eqref{eq:kato-K}, these equalities yield the proof of conjecture \eqref{eq:PR+} for $K_\infty^+/K$, from where the proof of Theorem~\ref{thm:CGS} can be deduced from Kato's work.

\subsection{From anticyclotomic to cyclotomic}

It remains to outline the proof of \eqref{eq:main-result}. 

Since Conjecture~\ref{conj:mazur-IMC} is known to be isogeny invariant, we replace $E$ by the elliptic curve $E_\bullet/\Q$ is the same isogeny class constructed by W\"uthrich \cite{wuthrich-int}. This can be characterized as the elliptic curve whose $p$-adic Tate module $T_pE_\bullet$ agrees with the geometric lattice in the $p$-adic representation $V_f$ realized as the maximal quotient of ${\rm H}^1_{\rm et}(Y_1(N)_{\overline{\Q}},\Q_p(1))$  on which the Hecke operators acts with the same eigenvalues as $f$.

Let ${\rm H}^1_{\rm Iw}(K_\infty,T_pE_\bullet)$ be the Iwasawa cohomology for the $\Z_p^2$-extension $K_\infty/K$, which by Shapiro's lemma can be identified with ${\rm H}^1(K,T_pE_\bullet\hat\otimes_{\Z_p}\Lambda_K)$. By the work of Lei--Loeffler--Zerbes and Kings--Loeffler--Zerbes, as refined in the case of interest in recent work of Burungale--Skinner--Tian--Wan, there exists a class
\[
{\rm BF}_\alpha\in{\rm H}_{\rm Iw}^1(K_\infty,T_pE_\bullet(\alpha))
\]
together with two explicit reciprocity laws:
\begin{itemize}
\item[(1)] At the prime $v$, the class ${\rm BF}_\alpha$ naturally lands in the subspace ${\rm H}^1(K_{v},T_p^+E_\bullet(\alpha))$ and there is a generalized Coleman power series map
\[
{\rm Col}_{v}:{\rm H}^1_{\rm Iw}(K_{\infty,v},T_p^+E_\bullet(\alpha))\hookrightarrow\Z_p^{\rm ur}\hat\otimes_{\Z_p}\Lambda_K
\]
sending ${\rm res}_{v}({\rm BF}_\alpha)$ to $\mathcal{L}_v^{\rm Gr}(f(\alpha)/K)$, where $\mathcal{L}_v^{\rm Gr}(f(\alpha)/K)$ is a two-variable Rankin--Selberg $p$-adic $L$-function with the property that its natural image  $\mathcal{L}_v^{\rm Gr}(f(\alpha)/K)^-$ in $\Lambda^{\rm ur}$ satisfies (as can be checked by comparing their respective interpolation properties)
\[
\bigl(\mathcal{L}_v^{\rm Gr}(f(\alpha)/K)^-\bigr)=\bigl(\mathcal{L}_v^{\rm BDP}(f(\alpha)/K)^2\bigr),
\]
where $\mathcal{L}_v^{\rm BDP}(f(\alpha)/K)$ is the twist of $\mathcal{L}_v^{\rm BDP}(f(\alpha)/K)$ by the anticyclotomic character $\alpha$.
\item[(2)] At the prime $\overline{v}$, there is a generalized Coleman power series map
\[
{\rm Col}_{\overline{v}}:{\rm H}_{\rm Iw}^1(K_{\infty,\overline{v}},T_p^-E_\bullet(\alpha))\hookrightarrow\Lambda_K,
\]
where $T_p^-E_\bullet(\alpha):=T_pE_\bullet(\alpha)/T_p^+E_\bullet(\alpha)$, sending the natural image of ${\rm BF}_\alpha$ to $\mathcal{L}_p^{\rm PR}(E(\alpha)/K)$.
\end{itemize}

The cyclotomic projection ${\rm BF}_\alpha^+\in{\rm H}^1_{\rm Iw}(K_\infty^+,T_pE_\bullet(\alpha))$ is the base class of a cyclotomic Euler system for $T_pE_\bullet(\alpha)$, and for a suitable choice of $\alpha$ it can be shown to be nonzero as a consequence of Rohrlich's nonvanishing results \cite{rohrlich-cyc} and the second of the above explicit reciprocity laws. By the Euler system machinery \cite{rubin-ES}, 
one thus obtains that a certain dual Selmer group $X_{{\rm ord},{\rm str}}(E_\bullet(\alpha)/K_\infty^+)$ (dual to the compact Selmer group ${\rm Sel}_{{\rm ord},{\rm rel}}(K_\infty^+,T_pE_\bullet(\alpha))$ on which the class ${\rm BF}_\alpha^+$ lives) is $\Lambda^+$-torsion, with characteristic ideal satisfying the divisibility
\[
{\rm char}_{\Lambda^+}\bigl(X_{{\rm ord},{\rm str}}(E_\bullet(\alpha)/K_\infty^+)\bigr)\supset{\rm char}_{\Lambda^+}\biggl(\frac{{\rm Sel}_{{\rm ord},{\rm rel}}(K_{\infty}^+,T_pE_\bullet(\alpha))}{\Lambda^+\cdot{\rm BF}_\alpha^+}\biggr)
\]
in $\Lambda^+[1/p]$. By the commutative hexagon deduced from  Poitou--Tate duality:

\begin{center}
\begin{tikzcd}
[row sep=large]
&\frac{{\rm H}^1_{\rm Iw}(K_{\infty,\overline{v}}^+,T_p^-E_\bullet(\alpha))}{\Lambda^+\cdot{\rm res}_{\overline{v}}({\rm BF}_\alpha^+)} \ar[r] & X(E_\bullet(\alpha)/K_\infty^+)\ar[twoheadrightarrow]{dr} &
\\ \frac{{\rm Sel}_{{\rm ord},{\rm rel}}(K_{\infty}^+,T_p^-E_\bullet(\alpha))}{\Lambda^+\cdot{\rm BF}_\alpha^+} \ar[hookrightarrow]{ur}{{\rm res}_{\overline{v}}}\ar[hookrightarrow]{dr}{{\rm res}_{{v}}} & & & X_{{\rm ord},{\rm str}}(E_\bullet(\alpha)/K_\infty^+) \\ & \frac{{\rm H}^1_{\rm Iw}(K_{\infty,v}^+,T_p^+E_\bullet(\alpha))}{\Lambda^+\cdot{\rm res}_{v}({\rm BF}_\alpha^+)} \ar[r] & X_v(E_\bullet(\alpha)/K_\infty^+) \ar[twoheadrightarrow]{ur} &
\end{tikzcd}
\end{center}
this translates into the divisibilities
\begin{equation}\label{eq:cyc-PR}
{\rm char}_{\Lambda^+}(X(E_\bullet(\alpha)/K_\infty^+))\supset{\rm char}_{\Lambda^+}\biggl(\frac{{\rm H}^1_{\rm Iw}(K_{\infty,\overline{v}}^+,T_p^-E_\bullet(\alpha))}{\Lambda^+\cdot{\rm res}_{\overline{v}}({\rm BF}_\alpha^+)}\biggr)
=\bigl(\mathcal{L}_p^{\rm PR}(E_\bullet(\alpha)/K)^+\bigr)
\end{equation}
with the equality following from the explicit reciprocity law at $\overline{v}$ (using that ${\rm Col}_{\overline{v}}$ 
has pseudo-null cokernel),
and
\begin{equation}\label{eq:cyc-Gr}
{\rm char}_{\Lambda^+}(X_v(E_\bullet(\alpha)/K_\infty^+))\tilde{\Lambda}^{+}\supset{\rm char}_{\Lambda^+}\biggl(\frac{{\rm H}^1_{\rm Iw}(K_{\infty,{v}}^+,T_p^+E_\bullet(\alpha))}{\Lambda^+\cdot{\rm res}_{{v}}({\rm BF}_\alpha^+)}\biggr)\tilde{\Lambda}^{+}
=\bigl(\mathcal{L}_v^{\rm Gr}(E_\bullet(\alpha)/K)^+\bigr),
\end{equation}
similarly using the explicit reciprocity law at $v$. 
Further choosing $\alpha$ so that $\mathcal{L}_v^{\rm BDP}(f/K)(0)\neq 0$ (as is possible by the nonvanishing of $\mathcal{L}_v^{\rm BDP}(f/K)$ as an element in $\Lambda^{\rm ur}$), we deduce from Theorem~\ref{thm:CGS} that both sides of the divisibility \eqref{eq:cyc-Gr} agree at $T=0$ and are nonzero, hence they are equal. From the commutative hexagon, it follows that the divisibility in \eqref{eq:cyc-PR} is also an equality, concluding the proof of \eqref{eq:main-result}. 


\bibliography{references}
\bibliographystyle{alpha}

\end{document}